\documentclass[article,12pt] {article}
\usepackage{graphicx}  
\usepackage[colorlinks=true, allcolors=blue]{hyperref}

\usepackage{amssymb}

\usepackage{amsmath}

\usepackage{amsthm}

\usepackage{amsmath,amsfonts,amssymb}
\usepackage{graphicx}
\usepackage{geometry}
\usepackage[numbers]{natbib}
\usepackage{hyperref}

\usepackage{xcolor}
\usepackage[table]{xcolor}
\usepackage{amsthm}
\usepackage{algorithm}
\usepackage{algpseudocode}
\usepackage{comment}
\usepackage{caption}
\usepackage{subcaption}
\usepackage{booktabs}
\usepackage{multirow}
\usepackage{tcolorbox}

\newtheorem{theorem}{Theorem}
\newtheorem{definition}{Definition}
\newtheorem{proposition}{Proposition}
\newtheorem{remark}{Remark}

\usepackage{bm}
\usepackage[normalem]{ulem}
\usepackage{lineno}
\usepackage{soul}

\usepackage{authblk} 
\newenvironment{keyword}{\noindent\textbf{Keywords: }}{}

\title{Accuracy and stability of Artificial Neural Networks for HP-Splines 
frequency parameter selection} 

\author[2]{Vittoria Bruni}
\author[1]{Paola Erminia  Calabrese}
\author[1]{Rosanna Campagna}
\author[2]{Domenico Vitulano}

\affil[1]{Department of Mathematics and Physics, University of Campania ``L. Vanvitelli'', Italy}
\affil[2]{Dept. of Basic and Applied Sciences for Engineering, Sapienza University of Rome, Italy}
\date{}

\begin{document}

\maketitle

\begin{abstract}
 This paper explores the use of artificial neural networks  for the  stable and data-driven selection of the frequency parameter in hyperbolic polynomial penalized splines (HP-splines). This parameter defines the underlying spline space and is essential for adapting the model to exponential patterns in the data, such as those encountered in signal processing. The theoretical approximation properties of deep neural network architectures are investigated to establish a connection between classical spline-based regression and modern data-driven learning methods.  Based on this analysis, a neural network is designed to predict optimal HP-spline parameters  by balancing approximation accuracy, stability analysis, and complexity control, thereby producing neural architectures that are both expressive and stable. Numerical experiments confirm that the proposed approach achieves both high accuracy and stable performance, validating the theoretical findings. 
\end{abstract}

\section{Introduction}
Multiexponential analysis is a fundamental tool for various research fields and application domains, including remote sensing, antenna design,  digital imaging \cite{CotroneiSauer,Tyler}, and Nuclear Magnetic Resonance analysis \cite{CCC19,ARomano2019}. Regularization techniques are usually required when estimating multiexponential decay models to address ill-conditioning, especially in the presence of noise. In this context, regression models based on exponential polynomial splines \cite{Romani2017} have been shown to provide stable and flexible approximations.
Hyperbolic-Polynomial Penalized splines (HP-splines) \cite{campagnaconti01,campagnaconti02} were specifically designed for data exhibiting exponential trends. They extend the polynomial space of the widely used regression P-spline  \cite{sicilianoPsplinePortfolio2018} to a hyperbolic-polynomial space defined by a frequency parameter $\alpha$. 
This parameter defines the underlying spline space $\mathbb{E}_{4,\alpha}$ and acts as a shape parameter for the resulting approximation. Indeed, the choice of $\alpha$ directly influences the ability of the hyperbolic B-spline basis (HB-splines) to reproduce exponential trends and to adapt to the structural behavior suggested by the data. As a consequence, different values of $\alpha$ may lead to significantly different approximations.
On the one hand, $\alpha$ controls the expressiveness of the approximation space; on the other hand, it affects the numerical stability and sensitivity of the associated penalized least-squares problem. These aspects are closely related, as inappropriate values of $\alpha$ may lead to poor data fitting or excessive amplification of data perturbations \cite{brunicampagnavitulano,brunicampagnavitulano02}.
This motivates the need for reliable and efficient strategies for its selection.
In \cite{campagnaconticuomo}, a data-driven algorithm is proposed for selecting the HP-spline frequency parameter based on the fitting dataset.
The algorithm relies on conditioning and stability results, choosing $\alpha$ by minimizing an estimate of the condition number associated with the Tikhonov formulation of the penalized model. This approach exploits explicit bounds on the sensitivity of the solution with respect to data perturbations, leading to an efficient criterion that simultaneously select $\alpha$ and the regularization parameter.
While effective, this procedure requires repeated evaluations over a predefined range of candidate frequencies and relies on model-dependent analytical estimates.
These challenges become more pronounced in bivariate multiexponential analysis \cite{Tyler}, particularly when extending P-splines via tensor products \cite{calabresecampagnaconti}. In this setting, penalty and frequency parameters must be selected along each orthogonal grid line, and estimating frequency parameters in multiple directions significantly increases computational cost and complexity compared to the univariate case. 
These issues motivate the development of alternative strategies for frequency parameter selection, capable of reducing the computational burden while enhancing flexibility. 

 Neural networks provide a reliable approach for estimating parameters that characterize functional spaces \cite{Lenzi,Crisci}.
Building on approximation results for shallow and deep networks with piecewise linear activation functions \cite{cybenko,Yarotsky2017}, we model the frequency parameter as a multivariate function of the data and design a neural network to approximate it, with complexity tuned to achieve a prescribed accuracy. In particular, the network depth and size are selected according to a specified  approximation error in the underlying functional space. We also derive bounds for the resulting HP-spline approximation and analyze the stability of the learning process.  {Generalization results  and a wavelet-based Lipschitz constant for the overall system are given}. Numerical experiments confirm the theoretical findings. 

The paper is organized as follows.
Section \ref{sec:formulation} briefly reviews the HP-spline formalism, along with key definitions and results on function approximation using \textit{ReLU}-based deep networks. The proposed model is then presented.
In Section \ref{sec:accuracy}, we derive upper bounds for the HP-spline approximation error when the frequency parameter is predicted by the  {artificial neural network}, under suitable assumptions on the functional space of the multivariate $\alpha$ function. This analysis enhances the interpretability of the network \cite{Yang2021}.
Section \ref{sec:stabilitybound} introduces stability results for a  {two-stage estimation model, in which a neural network predicts the parameters of a secondary parametric model}. A wavelet-based generalization bound is derived via uniform stability, and the effective Lipschitz constant of the overall system is computed explicitly. The resulting stability theorem is then stated.
Section \ref{sec:4} presents extended numerical experiments confirming the reliability of the proposed framework for frequency parameter estimation  {and HP-spline reconstruction}. Comparisons are performed with the deterministic estimation algorithm in \cite{campagnaconticuomo}. Conclusions and directions for future research close the paper.

\section{Problem Formulation}\label{sec:formulation}
This section introduces the proposed model for estimating the HP-spline frequency parameter. 
We first recall the definition and main properties of HP-splines, together with basic notions on neural network complexity and relevant function approximation results.

\subsection{HP-splines  {basics}}
Following the construction proposed in \cite{campagnaconti01}, we briefly give the definition of hyperbolic B-spline bases and the associated  HP-spline model.

\begin{definition}[Hyperbolic B-spline basis]
Let $\{\xi_1,\ldots,\xi_m\}$ be a uniform knot sequence with step size $h$, and let $\alpha \in \mathbb{R}$.  
The hyperbolic B-spline basis $\{B_j^{\alpha}\}_{j=0}^{m+1}$ is defined as a set of compactly supported, $C^2$-continuous spline functions whose local segments belong to the four-dimensional exponential space
\[\mathbb{E}_{4,\alpha} := \mathrm{span}\{e^{\alpha t},\; t e^{\alpha t},\; e^{-\alpha t},\; t e^{-\alpha t}\}.
\]
Each hyperbolic B-spline is supported on five consecutive knots and exhibits a bell-shaped profile, similarly to classical cubic B-splines.
\end{definition}

Following the P-spline paradigm, hyperbolic B-splines defined on a set of uniformly distributed knots, are combined with a second order penalty term, that  relaxes the importance of the location of the knots and their number. 

\begin{definition}[Hyperbolic-Polynomial Penalized Spline]
Given a set of  points $\{(t_i,y_i)\}_{i=1}^d$, a hyperbolic B-spline basis $\{B_j^{\alpha}\}_{j=0}^{m+1}$ defined on uniform knots, and a regularization parameter $\lambda>0$, the Hyperbolic-Polynomial Penalized spline (HP-spline) is the function
\begin{equation}\label{eqs}
s_{\alpha,\lambda}(t)=\sum_{j=0}^{m+1} a_j B_j^{\alpha}(t),
\end{equation}
whose coefficients $\{a_j\}$ are computed by solving the penalized least-squares problem
\[
\min_{a_0,\ldots,a_{m+1}}
\sum_{i=1}^{d}\left(y_i-\sum_{j=0}^{m+1} a_j B_j^{\alpha}(t_i)\right)^2
+ \lambda^2 \sum_{j=2}^{m+1}\left((\Delta^{h,\alpha}_2 \mathbf{a})_j\right)^2,
\]
where the $\alpha$-dependent second-order difference operator is defined as
\[
(\Delta^{h,\alpha}_2 \mathbf{a})_j
= a_j - 2 e^{-\alpha h} a_{j-1} + e^{-2\alpha h} a_{j-2}.
\]
The parameter $\lambda$ balances data fidelity and smoothness, while the parameter $\alpha$ controls the exponential behavior embedded in the spline basis.
\end{definition}

\subsection{Approximation properties of ReLU networks}
We briefly recall the main definitions and properties for feedforward neural networks with \textit{ReLU} activation function as formulated in \cite{Anthony2009} and later in \cite{Yarotsky2017}.
\begin{definition}
A feedforward neural network with \textit{ReLU} activation consists of input units, one output unit, and a collection of hidden computational units organized in layers. The inputs of each unit are the outputs of units from previous layers. Each computational unit maps an input vector $(x_k)_{k=1}^N$ to an output of the form
\[
y = \sigma\!\left(\sum_{k=1}^{N} w_k x_k + b \right),
\qquad \sigma(x) = \max(0,x),
\]

where $(w_k)_{k=1}^N$ and $b$ are trainable parameters, and the quantity $\sum_{k=1}^{N} w_k x_k + b$ is an affine combination of the inputs. The output unit is linear, i.e.,
\[
y = \sum_{k=1}^{N} w_k x_k + b.\]
\end{definition}\label{def:relu}

\begin{definition}
The complexity of a feedforward neural network is measured by its depth ($L$), the number of computational units ($U$), and the total number of weights ($W$). 
The depth is defined as the number of all layers. The total number of weights is given by the sum of all connections and computational units, where connections correspond to the coefficients in the linear combinations and each computational unit contributes one bias parameter.
\end{definition}\label{def:complexity}

We now recall a fundamental approximation result for \textit{ReLU} networks due to \cite{Yarotsky2017}.

In that work, functions are considered in the Sobolev space $\mathcal{W}^{k,\infty}([0,1]^d)$, and the approximation is measured in the infinity norm.

\begin{theorem}[Yarotsky, 2017]\label{thm:t1}
Let $d,k \in \mathbb{N}$ and $\varepsilon \in (0,1)$. Consider the unit ball
\[
\mathcal{F}_{k,d} = \left\{ f \in \mathcal{W}^{k,\infty}([0,1]^d) : \|f\|_{\mathcal{W}^{k,\infty}} \le 1 \right\}.
\]
Then there exists a \textit{ReLU} network architecture such that:
\begin{itemize}
\item it can approximate any $f \in \mathcal{F}_{k,d}$ with error at most $\varepsilon$ in the uniform norm;
\item its depth is bounded by $c \, (\log(1/\varepsilon)+1)$;
\item its size (number of weights and computational units) is bounded by
\[
c \, \varepsilon^{-d/k} (\log(1/\varepsilon)+1),
\]
for some constant $c = c(d,k)$.
\end{itemize}
\end{theorem}

This result illustrates the efficiency of deep \textit{ReLU} networks in approximating smooth functions, showing that both depth and size scale in a controlled way with the desired approximation accuracy.

\subsection{The proposed model}
Let $s(t)$ be a real-valued function depending on the time variable $t$.
Consider ${\bm s}$ a discrete sampling of $s(t)$ at some time points ${\bm t}=(t_1,\ldots,t_d)$,
and $ \alpha $ the frequency parameter of the fitting HP-spline (\ref{eqs}).   
 {$ \alpha $ can be computed using the data-driven algorithm in \cite{campagnaconticuomo} which associates a suitable value of $\alpha$ with the input signal ${\bm s}$. This naturally leads to modeling $\alpha $ as a multivariate function of the sampled data, i.e.,
  \begin{equation}
      \label{alphamap}
  \alpha:  {\bm s}  \in  \mathcal{S}  \subset \mathbb{R}^d   \to \mathbb{R} .
  \end{equation}
Denoting by $\mathrm{s_{hp}}(\alpha):=  
s_{\alpha,\lambda}$ 
the HP-spline  
{fitting the samples}  ${\bm s} : = (s_k)_{k=1,\ldots,d}$, $s_k=s(t_k)$, $k=1,\ldots,d$,
we aim to approximate this mapping using an artificial neural network.

Specifically, let ${\cal T} = \{({\bm s}_{i},\alpha_i)\}_{i=1,...,n}$ be the {\em training set} composed of $n$ sampled signals ${\bm s}_{i}$ and the corresponding frequency parameters $\alpha_i$, and let
\begin{equation}
      \label{Fthetamap}
F_\theta : \mathbb{R}^d \to \mathbb{R},
 \end{equation}
be the  \textit{ReLU} network with trainable parameters vector
${\bm \theta}$, trained on ${\cal T}$ based on the empirical training error  
\[
\widehat{\mathcal{E}}(F_\theta)
=
\frac{1}{n}
\sum_{i=1}^n
\ell\big(F_\theta({\bm s}_{i}), \alpha_i\big.),
\] 
where the loss function $\ell:   \mathbb{R} \times  \mathbb{R} \to  \mathbb{R}^+$ 
quantifies the difference between the model predictions and the actual target value.
Hence, for a given input sequence ${\bm s}$, the final goal is  to find a suitable \textit{ReLU} network   $F_\theta$ that predicts the frequency parameter $\alpha$ defining the HP-spline that best approximates the function ${s}(t)$ within a predefined tolerance $tol$, i.e.
\[
\| { s} - \mathrm{s_{hp}}(F_\theta({\bm s})) \|{_{\infty}} < tol .
\]

The computational flow can be synthesized as follows:
\[
{\bm s}
\;\overset{(1)}\longrightarrow\;
\alpha:=F_\theta({\bm s})
\;\overset{(2)}\longrightarrow\;
\mathrm{s_{hp}}({\alpha})=\mathrm{s_{hp}}(F_\theta({\bm s})).
 \]
 
The overall accuracy is determined by step $(1)$, which is governed by the  {neural network} prediction error for the frequency parameter, and step (2), which evaluates the accuracy of the signal reconstruction using the HP-spline $\mathrm{s_{hp}}(F_\theta({\bm s}))$, which belongs to the spline space determined by $\alpha$.
Based on this formulation, both accuracy and stability analysis rely on theoretical results on the approximation properties of deep neural network architectures investigated in \cite{Yarotsky2017} and related works, aiming to connect approximation in functional spaces with modern data-driven learning approaches.  
 
\section{Reconstruction error bound}\label{sec:accuracy}
This section provides bounds for signal reconstruction error of the model described in Section 2.3. To this aim, ${\alpha}$, as defined in (\ref{alphamap}), is assumed {to belong to}  a suitable Sobolev space $\mathcal{W}^{k,\infty}$ . The HP-spline is assumed to be defined on the spline space based on   $\mathbb{E}_{4, {\alpha}}$. 

\begin{proposition}
\label{prop:propapprox}
Let {{${\bm s} \in \mathcal{S} \subset \mathbb{R}^d$}} be a sampled signal
and $s_{hp}(\alpha)$  denote the   HP-spline   defined on $\mathbb{E}_{4,\alpha}$ and fitting the signal sample ${\bm s}$. 
Let us assume:
\begin{enumerate}
    \item 
    $\alpha({\bm s})\in  \mathcal{W}^{k,\infty}(\mathcal{S} \setminus E)$ with $E \subset \mathcal{S}$ a finite set of singular points and $\|\alpha\|_{\mathcal{W}^{k,\infty}}\leq 1$.
    \item $s_{hp}(\alpha)$ locally Lipschitz in $\alpha$ with constant $L_{\text{hp},i}$ in each smooth region $\mathcal{R}_i$, and possibly singular points in $E_{\text{hp}}$, and $L_{\text{hp}} = max_i L_{\text{hp},i}$.
    \item There exist ranges  of $\alpha$ such that
    $$\| s_{hp}({\alpha_1}) - s_{hp}({\alpha_2}) \|_{L^\infty} \le L_{\text{flat}} |\alpha_1 - \alpha_2|, \quad L_{\text{flat}} \ll L_{\text{hp}}.$$
\end{enumerate}
Let $F_\theta$ be a \textit{ReLU} network that approximates the frequency parameter $\alpha$.
For any $\varepsilon > 0$, there exists a \textit{ReLU} network $F_\theta$ of depth
$$L = O(\log(1/\varepsilon))$$
and size
$$W = O\big(\varepsilon^{-d/k} \log(1/\varepsilon) + |E \cup E_{\text{hp}}|\big),$$
such that  $F_\theta$ predicts $\alpha$.
  The reconstruction error 
  on the signal satisfies
\begin{equation}
\begin{aligned}
\| s - s_{\text{hp}}(F_\theta({ \bm s})) \|_{L^\infty} \le\;
&\underbrace{\| s - s_{\text{hp}}(\alpha({\bm s})) \|_{L^\infty}}_{\text{spline approximation error}} 
+ \underbrace{L_{\text{eff}} \, \varepsilon}_{\text{network error on smooth/flat regions}} \\
&+ \underbrace{C \sum_{{\bm s}_i \in E \cup E_{\text{hp}}} |F_\theta({\bm s}_i) - \alpha({\bm s}_i)|}_{\text{contribution from singular points}},
\end{aligned}
\end{equation}
with $L_{\text{eff}}$ and $C$ positive constants.
\end{proposition}

\begin{proof}
The proof refers to the different contributions to the signal approximation error separately.
\begin{description}
\item {\em Approximation of $\alpha(s)$ on smooth regions} 

On each smooth region $\mathcal{R}_i \subset \mathcal{S} \setminus E$, $\alpha(s) \in \mathcal{W}^{k,\infty}$ and $\|\alpha\|_{\mathcal{W}^{k,\infty}}\leq 1$. By Theorem \ref{thm:t1}, for any $\varepsilon > 0$ there exists a \textit{ReLU} network $F_\theta$ of depth $O(\log(1/\varepsilon))$ and size $O(\varepsilon^{-d/k} \log(1/\varepsilon))$ satisfying
$$\sup_{{\bm s} \in \mathcal{R}_i} |F_\theta({\bm s}) - \alpha({\bm s})| \le \varepsilon.$$

\item {\em Error propagation in locally Lipschitz and flat regions} 

For ${\bm s} \in \mathcal{R}_i$, we distinguish two cases:
\begin{enumerate}
    \item Lipschitz region: 
    $$    \| s_{\text{hp}}(\alpha({\bm  s})) - s_{\text{hp}}(F_\theta({\bm  s})) \|_{L^\infty} \le L_{\text{hp},i} | \alpha({\bm  s}) - F_\theta({ \bm s})| \le L_{\text{hp}} \varepsilon.
    $$

\item Flat region: 
    $$
    \| s_{\text{hp}}(\alpha({\bm s})) - s_{\text{hp}}(F_\theta({\bm s})) \| \le L_{\text{flat}} \varepsilon \ll L_{\text{hp}} \varepsilon.
    $$
     
\end{enumerate}
Define $L_{\text{eff}} = \max(L_{\text{hp}}, L_{\text{flat}})=L_{\text{hp}}$ as a conservative effective Lipschitz constant.

\item {\em Singular points} 

$ \forall \; {\bm s}_i \in E \cup E_{\text{hp}}$, define $\delta_i = |F_\theta({\bm s}_i) - \alpha({\bm s}_i)|$ and let $C > 0$ such that
$$
\| s_{\text{hp}}(\alpha({\bm s}_i)) - s_{\text{hp}}(F_\theta({\bm s}_i)) \| \le C \, \delta_i.
$$

\end{description}

Using the triangle inequality we can combine all previous errors to bound signal approximation error:
\begin{equation}
\begin{split}
\| s - s_{\text{hp}}(F_\theta({\bm s})) \|_{L^\infty} &\le \| s - s_{\text{hp}}(\alpha({\bm s})) \|_{L^\infty} + \| s_{\text{hp}}(\alpha({\bm s})) - s_{\text{hp}}(F_\theta({\bm s})) \|_{L^\infty} \\
&\le \| s - s_{\text{hp}}(\alpha({\bm s})) \|_{L^\infty} + L_{\text{eff}} \varepsilon + C \sum_{{\bm s}_i \in E \cup E_{\text{hp}}} \delta_i.
\end{split}
\end{equation}

Regarding the network size, we have to account for the  contributions of  smooth regions and singular points. From  Theorem \ref{thm:t1}, smooth regions contribute with $O(\varepsilon^{-d/k} \log(1/\varepsilon))$ parameters, while  singular points contribute with $O(|E \cup E_{\text{hp}}|)$ parameters, that are necessary for local corrections.  
Thus, the total network size is 
$$W = O(\varepsilon^{-d/k} \log(1/\varepsilon) + |E \cup E_{\text{hp}}|),$$ 
with depth $L = O(\log(1/\varepsilon)) $.

\end{proof}

\begin{remark}

It is worth observing that the network error on smooth/flat regions is the one provided in  \cite{Yarotsky2017}. Moreover, even if $\,|F_\theta({\bm s}) - \alpha({\bm s})|$ is large, the reconstruction error $\| s - s_{\text{hp}}(F_\theta({\bm s})) \|$ may remain small in regions where $s_{\textit{hp}}(\alpha)$ is insensitive to $\alpha$. Hence, the network may estimate $\alpha$ that deviates from nominal values while still producing near-optimal reconstruction.
\end{remark}

In the next section we analyze network generalization, by evaluating how well $F_\theta({\bm s})$ 
predicts $\alpha:=\alpha({\bm s})$ on unseen signals.

 \section{Stability bounds for neural networks outputting model parameters} 
\label{sec:stabilitybound}
In this section, based on uniform stability  results, we derive a generalization bound  for the neural network by considering the proposed \textit{two-stage learning procedure} in which a neural network outputs the parameter  of a secondary parametric model. Moreover, we introduce  a finer  wavelet-based generalization bound. The effective Lipschitz constant of the full system is expressed explicitly, and the resulting stability theorems are stated formally. 

Stability evaluates how sensitive the learned model is to small changes in the training dataset, and uses this to control the gap between training and test performance.
To define stability bounds for the learning model (step 1 of the proposed procedure), we recall the formal definition of uniform stability  \cite{Bousquet2002}.\\

\begin{definition}[Uniform Stability]
Let $A: {\cal T} \subset \mathbb{R}^{n\times d} \to  \Theta$ be a learning algorithm mapping a dataset ${\cal T}$ to a model in $ \Theta$ and
$\ell : \Theta \times Z \to  \mathbb{R}$ be a function such that $\ell(f, z)$  measures the loss of model $f$ on point $z\in Z\subset \mathbb{R}$.  The learning algorithm $A$ is said to be $\beta$-uniformly stable if for all datasets
	${\cal T}=({\bm s}_1,\dots,{\bm  s}_n)$,
	all modified datasets ${\cal T}^{(i)}$ obtained by replacing ${\bm s}_i$ with ${\bm s}'_i$, and for all points $z$,
	\begin{equation}
	    \label{gap} |\ell(A({\cal T}),{z}) - \ell(A({\cal T}^{(i)}),{z})| \le \beta.
	\end{equation}
\end{definition}

Uniform stability then controls the sensitivity of the learned hypothesis to the replacement of a single training example. Hence, if $\beta$ is small, the network $F_\theta$ generalizes well.

In order to define a  \textit{stability-based generalization bound} for the proposed two-stage learning procedure, we introduce the following definitions
\cite{Feldman2018}.
The  \textit{empirical loss}  of a learning algorithm $A$ on the dataset ${\cal T}$ is defined as
\[{\mathcal{E}_{\cal T}}[\ell(A({\cal T}))] = \frac{1}{n} \sum_{i=1}^n \ell(A({\cal T}),{ z}_i);\] 
 the  {\textit{expected loss}} relative to the distribution ${\mathcal D}$ on $Z$ is
 \[
 {\mathcal{E}_{\mathcal D}}[\ell(A({\cal T}))] =\mathbb{E}_{Z\sim\mathcal D}
 \big[
  \ell(A({\cal T}),z)
 \big].
 \]
Finally,    the  {\textit{estimation error}} (also referred to as the \textit{generalization gap}) of $A$ on $S$ relative to  ${\mathcal D}$ is defined as 
\begin{equation}
  \label{eqDelta}  \Delta_{\cal T}(\ell(A))={\mathcal{E}_{\mathcal D}}[\ell(A({\cal T})]-{\mathcal{E}_{\cal T}}[\ell(A({\cal T})]
\end{equation}  

Following the arguments in \cite{Feldman2018}, 
if the learning algorithm is $\beta-$stable, then 
 {\textit{expected generalization gap}} can  be bounded as follows:
$$    |\mathbb{E}[\Delta_{\cal T}(\ell(A))]| \leq \mathcal{O}\!\left(\beta+ \frac{1}{\sqrt{n}} \right)$$
and it is observed to be   optimal when
$\beta= \mathcal{O}(1/n)$. In the next proposition we consider $\beta< \mathcal{O}(1/\sqrt{n})$ and 
\begin{equation}
    \label{gengapp}
    GenGap:=|\mathbb{E}[\Delta_{\cal T}(\ell(A))]| \leq \mathcal{O}\!\left(  \frac{\beta}{\sqrt{n}} \right)
.
\end{equation}
   
 Given these assumptions, we prove the following propositions about 
 the stability of a learning algorithm for the HP-spline frequency parameter.

\begin{proposition}
\label{prop:2}    
 Let ${\mathcal T} = \{({\bm s}_i,\alpha_i)\}_{i=1,...,n}$ be a training set, $F_\theta : \mathcal S\subset \mathbb{R}^d \to \mathbb{R}$
 a neural network  such that ${\alpha}:= F_{\theta}({\bm s})  $  under the assumptions in Proposition \ref{prop:propapprox},
 let $s_\alpha(\cdot,\alpha)$  be the composite function acting both on the sampled signal and on the  frequency parameter, to estimate the HP-spline coefficients from $({\bm s},F_{\theta}({\bm s}))$,
  $ s_{{\alpha}} :     \mathbb{R}^{d \times 1} 
     \to \mathbb{R}^{m+2}$,
      and let  $F_{\theta}(\mathcal{T})$ denote  the learning model using ${\cal T}$.  
If the neural network $F_\theta$ is $L_F-$Lipschitz, i.e.
\begin{equation}
    \label{hyp}
\|F_\theta({\bm s})-F_\theta({\bm s}')\| \le L_F \|{\bm s}-{\bm s}'\|, \quad  {\bm s},{\bm s}'\in\mathcal S
\end{equation}
 and 
the composite model 
satisfies the two Lipschitz- conditions with respect to both its first and second argument, i.e.,
\begin{align*}
\|s_\alpha({\bm s},\alpha)-s_\alpha({\bm s}',\alpha)\| &\le L_{\bm s} \|{\bm s}-{\bm s}'\|,\quad  {\bm s},{\bm s}'\in\mathcal S \\
 {\|s_\alpha({\bm s},\alpha)-s_{\alpha'}({\bm s},\alpha')\| }& {\le L_\alpha \|\alpha-\alpha'\|\quad  \alpha, \alpha'\in \mathbb{R},
}\end{align*}
then  it holds
\begin{align}
\label{salphaL}
\|s_\alpha ({\bm s},\alpha) -{\bm s}_{\alpha'} ({\bm s}',\alpha') \|& \le (L_{\bm s} + L_\alpha L_F)D,
\end{align} 
with $L_{\bm s}$, $L_\alpha$    and  $L_F$ Lipschitz constants and
$D = \sup_{{\bm s},{\bm s}'\in\mathcal S} \|{\bm s}-{\bm s}'\|$
the diameter of the input space $\mathcal S$.
\end{proposition}

\begin{proof}
Let ${\bm s},{\bm s}'\in \mathcal S\subset \mathbb{R}^d$,  ${\alpha}= F_{\theta}({\bm s})$, ${\alpha}'=  F_{\theta}({\bm s}')$, then 
\begin{align*}
\|s_\alpha ({\bm s},\alpha) -s_{\alpha'} ({\bm s}',\alpha') \|
&\le
\|s_\alpha ({\bm s},\alpha) -s_{\alpha} ({\bm s}',\alpha) )\|  +
\|s_\alpha ({\bm s}',\alpha) -s_{\alpha'} ({\bm s}',\alpha') \|\\
&\le L_{\bm s} \|{\bm s}-{\bm s}'\|+L_\alpha \|F_\theta({\bm s})-F_\theta({\bm s}')\| .
\end{align*}
Since $F_\theta$ is $L_F$-Lipschitz we get
\begin{align*}
\|s_\alpha ({\bm s},\alpha) -s_{\alpha'} ({\bm s}',\alpha') \|& \le (L_{\bm s} + L_\alpha L_F)\|{\bm s}-{\bm s}'\|.
\end{align*} 
from which the assertion (\ref{salphaL}) follows.
\end{proof}

Therefore, when a neural network outputs parameters of a secondary model, the effective Lipschitz constant of the overall architecture decomposes into two contributions, under suitable assumptions on the parametric model.
This result allows to get a generalization bound.

\begin{proposition}
\label{prop:3}
 Under the same assumptions of Proposition \ref{prop:2},
 if the loss function  is 
 $L_\ell-$Lipschitz in its first argument,
then the learning algorithm is $\beta$-uniformly stable  with
\begin{equation}
    \label{betabound}
\beta \leq L_\ell   L_F  D, \quad D = \sup_{{\bm s},{\bm s}'\in\mathcal S} \|{\bm s}-{\bm s}'\| 
\end{equation}
and the expected loss relative to a distribution $\mathcal{D}$   on $\mathbb{R}$ is
\begin{equation}
    \label{generalizationbound1}
	|\mathbb{E}[\ell(F_\theta(\mathcal{T}),\alpha)]|
	\;\le\;
	\hat{\mathcal{E}}(F_\theta)
	+
	\mathcal{O}\!\left( \frac{L_\ell   L_F   D}{\sqrt{n}} \right).
	\end{equation}
    with 
\begin{equation}
    \label{hate}\hat{\mathcal{E}}(F_\theta):= 
  \frac{1}{n} \sum_{i=1}^n \ell(F_\theta(\mathcal{T}),\alpha_i).   
  \end{equation}
  Similarly, $\mathcal{O}\!\left( \frac{L_\ell   (L_{\bm s} + L_\alpha L_F D)}{\sqrt{n}} \right)$ bounds the expected generalization gap for the composite function $s_{\alpha}$.
\end{proposition}
     
\begin{proof}
Let ${\bm s},{\bm s}'\in \mathcal S\subset \mathbb{R}^d$,  ${\alpha}=  F_{\theta}(\mathcal{T})$, ${\alpha}'=  F_{\theta}(\mathcal{T}')$, with $\mathcal{T}$ and $\mathcal{T}'$ training sets defined according to Definition 5; then it holds
\begin{eqnarray}
&&|\ell(F_\theta({\mathcal T}),\alpha ) - \ell(F_\theta({\mathcal T}'),\alpha')|
\le L_\ell \|F_\theta({\mathcal T}) -F_\theta({\mathcal T}') \|\le\nonumber \\
&&
\le L_\ell L_F \|{\mathcal T}-{\mathcal T}'\|
\le L_\ell L_F D, \nonumber
\end{eqnarray}
with $D$ the diameter of the input space, and the (\ref{betabound}) is given.
To prove (\ref{generalizationbound1}), 
 we consider the   expected loss for the learning network $F_\theta(\mathcal{T})$
     relative to the distribution $\mathcal{D}$:
\[  {\mathcal{E}_{\mathcal D}}[\ell(F_\theta(\mathcal{T}))]= {\mathcal{E}_\mathcal{T}}[\ell(F_\theta(\mathcal{T}))]+\Delta_\mathcal{T}(\ell(F_\theta))\]
\[  |{\mathcal{E}_{\mathcal D}}[ \ell(F_\theta(\mathcal{T}))]|\leq |{\mathcal{E}_\mathcal{T}}[\ell(F_\theta(\mathcal{T}))]|+|\Delta_\mathcal{T}(F_\theta(\mathcal{T}))| \]
Using (\ref{gengapp}) we finally have
$$  |\mathbb{E}[\ell(F_\theta(\mathcal{T}),\alpha)] |
\leq  \hat{\mathcal{E}}(F_\theta) + \mathcal{O}\left( \frac{L_\ell  L_F  D}{\sqrt{n}} \right)$$ 
    
 with  $  \hat{\mathcal{E}}(F_\theta)$ defined as in (\ref{hate}). Since the learning algorithm is $\beta-$ stable, using eq. (\ref{betabound}) the thesis follows. 

 Using similar arguments and Proposition \ref{prop:2}, we get the bound for the expected generalization gap for the composite function $s_{\alpha}$.
\end{proof}

The generalization bound scales proportionally to the upper bound in (\ref{betabound}), revealing how both parameter sensitivity and input sensitivity affect stability and generalization bound. In particular, if ${\mathcal L}:=L_\ell  L_F$ and $n$ are fixed, reducing the diameter of the input space reduces the generalization gap.
It is then expected that applying a non expansive transform to the input sequences reduced the diameter and then the expected generalization gap.
Considering its property, a candidate transform is the wavelet transform.

Let $\{\phi_{J,k}\}$ be the scaling functions of a multiresolution analysis \cite{Mallat1999} and let 
$$V_J s := \sum_k \langle s, \phi_{J,k} \rangle \phi_{J,k}$$
be the approximation operator at scale $J$.
For orthonormal or biorthogonal wavelets, $V_J$ is non-expansive, that is
$$\|V_J s - V_J s'\| \le \|s-s'\|.$$
By observing that the effective diameter of a set on input data after wavelet approximation is
$$D_J := \sup_{s,s'\in\mathcal{S}} \| V_J s - V_J s'\|,$$
it easily follows that projecting into a wavelet approximation space decreases the diameter and the following proposition holds true.

\begin{proposition}
\label{prop:4}
Under the same assumption of Proposition \ref{prop:2}, assume that the loss function of the learning algorithm is
 $L_\ell-$Lipschitz in its first argument, and let
 $V_J$ be a non-expansive wavelet projection at scale $J$,
	then it holds
	$$\mathbb{E}[\ell(F_\theta(V_J \mathcal{T}),\alpha)]
	\;\le\;
	\hat{\mathcal{E}}_J(F_\theta)
	+
	\mathcal{O}\!\left( \frac{L_\ell L_F D_J}{\sqrt{n}} \right),
	$$
    with $$\hat{\mathcal{E}}_J(F_\theta) = \frac{1}{n} \sum_{i=1}^n \ell(F_\theta(V_J \mathcal{T}_i),  \alpha_i).$$
	In particular, if $D_J \ll D$, where $D=\sup_{\mathcal{T},\mathcal{T}'}\|\mathcal{T}-\mathcal{T}'\|$, then
	\[
	\text{GenGap}(J) \leq
	\mathcal{O}\!\left( \frac{D_J}{\sqrt{n}} \right)
	<
	\mathcal{O}\!\left( \frac{D}{\sqrt{n}} \right).
	\]
    with $\text{GenGap}(J)$ defined as in (\ref{gengapp}).
\end{proposition}

\begin{proof}
Using the Lipschitz properties of the loss function difference, under a one-sample replacement in the training set, we get
\begin{eqnarray}
&&|\ell(F_\theta(V_J {\mathcal T}),\alpha ) - \ell(F_\theta(V_J {\mathcal T}'),\alpha')|
\le L_\ell \|F_\theta(V_J {\mathcal T})-F_\theta(V_J {\mathcal T}')\|\le\nonumber \\
&&
\le L_\ell L_F \|V_J {\mathcal T}-V_J {\mathcal T}'\|
\le L_\ell L_F D_J.\nonumber
\end{eqnarray}
Thus the learning algorithm is $\beta$-stable with $\beta = L_\ell L_F D_J$, and the stability generalization bound follows. Using the non-expansion property of the wavelet operator,  $D_J \leq D$ providing a straightforward bound for the generalization gap.
\end{proof}

The result in Proposition \ref{prop:4} depends on the resolution of the approximation level $J$. On the other hand, if $s:\mathbb{R}^d \to \mathbb{R}$ has Hölder regularity $\gamma>0$ \cite{Mallat1999,Daubechies1992}, then
$$\|s - V_J s\|_{L^2} \le C \, 2^{-J \gamma}.$$
This motivates considerations on the choice of the truncation scale $J$, balancing the approximation error (bias) and the generalization properties associated with the model complexity. In particular, the optimal value of $J$ is expected to trade off approximation accuracy and generalization performance. Coarse scales (small $J$) reduce the diameter $D_j$, thus improving generalization, but may increase the approximation error. Conversely, finer scales (large $J$) improve approximation accuracy at the cost of increased model complexity and larger $D_j$.
Therefore, while coarse-scale approximations can enhance generalization through improved stability, excessively coarse representations may significantly degrade accuracy. This trade-off warrants further investigation.

\section{Numerical Results}
\label{sec:4}
This section validates the theoretical findings through numerical simulations.
The tests are executed on a personal computer equipped with a $2.80$ GHz Intel Core(TM) $i7-7700HQ$ CPU and $16.0$ GB of RAM, and the algorithms are implemented in  Matlab  R2025b.

Accuracy and stability are evaluated through standard vector-valued metrics:
\begin{itemize}
    \item the \textit{Mean Squared Error}
    $$MSE(\mathbf{v}, \hat{\mathbf{v}}) 
        = \frac{1}{d} \sum_{j=1}^{d} \big( v_j - \hat{v}_j \big)^2, \qquad \mathbf{v}, \hat{\mathbf{v}} \in \mathbb{R}^d$$
    quantifying the average squared deviation;
    \item the \textit{Infinity Norm Error}
    \[\|\mathbf{v} - \hat{\mathbf{v}}\|_\infty 
        = \max_{1 \le j \le d} |v_j - \hat{v}_j|, \qquad \mathbf{v}, \hat{\mathbf{v}} \in \mathbb{R}^d\]
    capturing the largest pointwise deviation, and
   \item the \textit{Relative Error}
            \[
RE(\mathbf{v},\hat{\mathbf{v}}) = \frac{\|\mathbf{v} - \hat{\mathbf{v}}\|_2}{\|\mathbf{v}\|_2}, \qquad \mathbf{v}, \hat{\mathbf{v}} \in \mathbb{R}^d\
\]
quantifying the relative deviation with respect to the magnitude of the reference vector.
\end{itemize}

\subsection{Parameter approximation and error propagation} \label{sec:test_accuracy}
The approximation results in \cite{Yarotsky2017} and Proposition~\ref{prop:propapprox} ensure the existence of neural network architectures with controlled complexity that can achieve arbitrarily small error in approximating $\alpha$  and the corresponding HP-spline reconstruction. However, these results are non-constructive and do not provide explicit design guidelines. Therefore, the theoretical analysis is complemented with an empirical investigation to evaluate the approximation capabilities of \textit{ReLU} networks.

The training and test datasets are constructed from a parameter function $\alpha(x)$ defined on the interval $x \in [0,1]$,  representing the quantity to be predicted by the network. The function is evaluated at discrete points $x_i$, producing the values $\alpha_i = \alpha(x_i)$. Each $\alpha_i$ is then used to generate network input data
\[s_{\alpha_i}(t) = e^{-\alpha_i t},\]
while $\alpha_i$ serves as the corresponding target output.  
The functions $\alpha(x)$ are selected to satisfy the theoretical assumptions in Proposition~\ref{prop:propapprox}.
Both smooth functions
\[\alpha_1(x) = \frac{1}{1+x}, \quad 
\alpha_2(x) = 1 + \frac{1}{2}\sin(2\pi x),\]
and functions with locally reduced regularity, such as
\[\alpha_3(x) = x^2 + |x-x_0|^{k-0.5}, \quad x_0 = 0.37, \ k=2\]
as well as the piecewise quadratic function
\[\alpha_4(x) =
\begin{cases}
x^2, & x \le \tfrac{1}{3},\\[1mm]
\frac{11}{3} x - \frac{10}{9}, & \tfrac{1}{3} < x \le 0.6,\\[1mm]
\frac{49}{45}, & 0.6 < x \le 0.65,\\[1mm]
3 t^2, & x > 0.65
\end{cases}\]
are considered to evaluate the network’s performance across different regularity conditions.

For each function, fully connected networks with fixed depth \(L \in \{3,4,5\}\) are trained, increasing the width \(W\) until the approximation condition \(\|\alpha({\bm s}) - F_\theta({\bm s})\|_\infty \le \varepsilon\) is satisfied. Each architecture has uniform width across layers. For each configuration, the corresponding network complexity \(C_{\mathrm{tot}}\) is reported, defined as in Section 2.2, depending on  \((L,W)\) and the input dimension \textcolor{blue}{$d$}, which is set to 32 in this study.

Table~\ref{tab:metrics} summarizes the results for both smooth functions \(\alpha_1, \alpha_2\) and functions with isolated singularities \(\alpha_3, \alpha_5\). For each \(\varepsilon\) and $L$, the table reports the architecture achieving the target precision, along with the corresponding complexity and error metrics. In particular, the reported metrics refer to parameter approximation errors, HP-spline reconstruction errors computed using both predicted and nominal parameters, and the propagation error, which measures the discrepancy between the reconstructed signals obtained with the predicted and nominal parameters, which quantifies the effect of approximation errors in \(\alpha\) on the reconstructed HP-spline signals.

\begin{table}[H]
\centering
\footnotesize
\setlength{\tabcolsep}{4pt}

\begin{tabular}{|c|ccc|cc|cc|cc|c|}
\hline
\multirow{2}{*}{$\epsilon$} &
\multicolumn{3}{c|}{Architecture} &
\multicolumn{2}{c|}{Parameter error} &
\multicolumn{4}{c|}{HP-spline Reconstruction error} &
\multicolumn{1}{c|}{Propagation} \\

\multirow{2}{*}{} &
\multicolumn{3}{c|}{} &
\multicolumn{2}{c|}{} &
\multicolumn{4}{c|}{} &
\multicolumn{1}{c|}{error} \\
\cline{2-11}

& $L$ & $W$ & $C_{\mathrm{tot}}$
& $Max_\alpha$ & $MSE_\alpha$
& $Max_{\mathrm{pred}}$ & $MSE_{\mathrm{pred}}$
& $Max_{\mathrm{nom}}$ & $MSE_{\mathrm{nom}}$
& $MSE$ \\
\hline
\multicolumn{11}{|c|}{$\mathbf{\alpha_1(x)}$} \\ \hline
0.10 & 3 & 1 & 35  & 3.67e-03 & 1.02e-06 & 1.84e-09 & 5.96e-22 & 1.36e-09 & 1.53e-22 & 4.46e-22 \\
0.07 & 3 & 2 & 69  & 2.35e-02 & 1.06e-04 & 1.82e-09 & 7.29e-22 & 1.55e-09 & 2.91e-22 & 4.46e-22 \\
0.008 & 3 & 3 & 103 & 6.24e-03 & 1.79e-06 & 3.91e-08 & 1.20e-19 & 3.91e-08 & 1.20e-19 & 4.46e-22 \\
0.10 & 4 & 3 & 115 & 2.02e-03 & 2.99e-07 & 1.73e-08 & 2.91e-20 & 1.73e-08 & 2.87e-20 & 4.46e-22 \\
0.07 & 4 & 3 & 115 & 1.75e-02 & 2.30e-05 & 4.46e-09 & 1.34e-21 & 4.45e-09 & 8.98e-22 & 4.46e-22 \\
0.008 & 4 & 2 & 75  & 2.01e-04 & 4.49e-09 & 1.82e-09 & 4.68e-22 & 5.91e-10 & 2.18e-23 & 4.46e-22 \\
0.10 & 5 & 2 & 81  & 1.37e-02 & 1.79e-05 & 3.41e-09 & 1.37e-21 & 3.40e-09 & 9.22e-22 & 4.46e-22 \\
0.07 & 5 & 1 & 39  & 5.28e-02 & 5.39e-04 & 1.86e-09 & 6.95e-22 & 1.45e-09 & 2.52e-22 & 4.46e-22 \\
\hline

\multicolumn{11}{|c|}{$\mathbf{\alpha_2(x)}$} \\ \hline
0.10 & 3 & 2 & 69  & 9.27e-04 & 2.89e-07 & 1.99e-10 & 6.03e-24 & 1.95e-09 & 2.74e-22 & 2.79e-22 \\
0.07 & 3 & 1 & 35  & 3.01e-03 & 2.21e-06 & 3.82e-10 & 1.31e-23 & 1.95e-09 & 2.74e-22 & 2.87e-22 \\
0.008 & 3 & 1 & 35 & 8.91e-04 & 3.27e-07 & 9.07e-10 & 5.57e-23 & 1.95e-09 & 2.74e-22 & 3.28e-22 \\
0.10 & 4 & 2 & 75  & 1.47e-02 & 2.07e-05 & 3.95e-09 & 7.71e-22 & 1.95e-09 & 2.74e-22 & 1.04e-21 \\
0.07 & 4 & 1 & 37  & 1.68e-02 & 1.08e-04 & 1.92e-08 & 8.41e-21 & 1.95e-09 & 2.74e-22 & 8.68e-21 \\
0.10 & 5 & 3 & 127 & 7.32e-02 & 5.77e-04 & 3.66e-09 & 6.98e-22 & 1.95e-09 & 2.74e-22 & 9.73e-22 \\
0.07 & 5 & 1 & 39  & 1.90e-02 & 6.56e-05 & 7.84e-10 & 4.50e-23 & 1.95e-09 & 2.74e-22 & 3.19e-22 \\
\hline

\multicolumn{11}{|c|}{$\mathbf{\alpha_3(x)}$} \\ \hline
0.10 & 3 & 1 & 35  & 4.20e-02 & 1.71e-04 & 2.43e-08 & 9.77e-20 & 3.36e-08 & 1.19e-19 & 2.05e-19 \\
0.07 & 3 & 2 & 69  & 5.82e-02 & 4.78e-04 & 1.57e-07 & 2.13e-18 & 3.36e-08 & 1.19e-19 & 2.25e-18 \\
0.10 & 4 & 3 & 115 & 6.39e-02 & 4.47e-04 & 1.14e-07 & 1.48e-18 & 3.36e-08 & 1.19e-19 & 1.60e-18 \\
0.07 & 4 & 3 & 115 & 5.15e-03 & 2.44e-06 & 4.86e-08 & 2.46e-19 & 3.36e-08 & 1.19e-19 & 3.64e-19 \\
0.10 & 5 & 5 & 231 & 4.55e-02 & 9.24e-05 & 4.06e-08 & 2.62e-19 & 3.36e-08 & 1.19e-19 & 3.80e-19 \\
0.07 & 5 & 2 & 81  & 4.19e-03 & 2.22e-06 & 1.55e-08 & 3.67e-20 & 3.36e-08 & 1.19e-19 & 1.56e-19 \\
\hline

\multicolumn{11}{|c|}{$\mathbf{\alpha_4(x)}$} \\ \hline
0.10 & 3 & 1 & 35  & 5.47e-02 & 2.69e-04 & 2.46e-08 & 1.11e-19 & 2.29e-07 & 8.53e-18 & 8.68e-18 \\
0.07 & 3 & 1 & 35  & 6.60e-03 & 7.00e-06 & 1.04e-06 & 1.41e-17 & 2.29e-07 & 8.53e-18 & 2.26e-17 \\
0.10 & 4 & 3 & 115 & 3.52e-02 & 1.29e-04 & 1.08e-07 & 1.49e-18 & 2.29e-07 & 8.53e-18 & 1.01e-17 \\
0.07 & 4 & 1 & 37  & 4.34e-03 & 2.12e-06 & 4.48e-07 & 1.50e-17 & 2.29e-07 & 8.53e-18 & 2.34e-17 \\
0.10 & 5 & 4 & 177 & 7.05e-02 & 3.99e-04 & 7.54e-07 & 5.61e-17 & 2.29e-07 & 8.53e-18 & 6.35e-17 \\
0.07 & 5 & 2 & 81  & 5.81e-02 & 2.22e-04 & 3.05e-08 & 8.84e-20 & 2.29e-07 & 8.53e-18 & 8.60e-18 \\
0.008 & 5 & 1 & 39 & 5.96e-03 & 3.20e-06 & 3.78e-06 & 1.66e-16 & 2.29e-07 & 8.53e-18 & 1.75e-16 \\
\hline
\end{tabular}

\caption{
Approximation results for the functions \(\alpha_1,\dots,\alpha_4\), including both smooth and singular cases. 
For each \(\varepsilon\), the table reports the network architecture \((L,W)\), the total complexity \(C_{\mathrm{tot}}\). 
The approximation errors are given in terms of infinity norm (\(Max_\alpha\)) and mean squared error (\(MSE_\alpha\)), 
HP-spline reconstruction errors are computed using  uniformly spaced knots with step size \(0.1\), 
and are reported for both predicted and nominal parameters. 
The propagation error measures the discrepancy between the corresponding reconstructed signals, quantifying the effect of approximating \(\alpha\).
} 
\label{tab:metrics}
\end{table}

For all $\alpha$ functions, the HP-spline reconstructions using the predicted parameters \(F_\theta(s)\) closely match those obtained with the nominal parameters \(\alpha(s)\). The propagation error remains consistently  smaller than the intrinsic spline error, indicating that inaccuracies in the neural approximation have minimal impact on the final reconstruction, even in the presence of few singular points. This demonstrates that moderately-sized \textit{ReLU} networks can accurately approximate the parameters  without compromising HP-spline reconstruction quality.

It should be noted that the  architectures in Table \ref{tab:metrics} provide only an empirical indication: failure to meet the prescribed tolerance for a given configuration does not preclude the existence of alternative networks of similar complexity that could achieve the same accuracy.

\subsection{Empirical evaluation of stability based on general gap bound}

In this section, we adopt an empirical approach to assess stability. We consider a neural network with \textit{ReLU} activation trained on a dataset ${\cal T}$. The test set is constructed as a slight perturbation of ${\cal T}$, obtained by adding noise to a limited fraction of the data. After training, we compute the loss on the training set, $Loss_{\mathrm{train}}$, and the loss on the test set, $Loss_{\mathrm{test}}$. The generalization gap is then defined empirically as
\[
\mathrm{GenGap} = \bigl|Loss_{\mathrm{train}} - Loss_{\mathrm{test}}\bigr|.
\]

In this setting, the theoretical results discussed in the Section \ref{sec:stabilitybound} suggest that the generalization gap should increase with the dataset diameter and decrease as the sample size grows. In particular, the expected behavior is 
\[
\mathrm{GenGap} \le O\!\left(\frac{D}{\sqrt{n}}\right),
\qquad
\mathrm{GenGap}_J \le O\!\left(\frac{D_J}{\sqrt{n}}\right),
\]
where $n$ and $D$ are the size and diameter of ${\cal T},$ respectively, and $D_J$ the corresponding diameter after wavelet projection at level $J$.  

In our experiments, we vary both the sample size and the dataset diameter. We consider synthetic signals of the form
\[
s(t) = A \sum_{k=1}^{5} 2^{-k\alpha}\cos(2^{k}\pi t), \qquad t \in [0,1],
\]
where $\alpha$ is randomly sampled in $[0.5,5]$, while the amplitude $A \in \{0.5,1,2,4.5,5,6,8,10.5\}$ controls the overall scale and hence the diameter.

The training dataset consists of $n$ signals, with $n \in [32, 64, 128, 256, 512]$, each sampled on a uniform grid of $256$ points in $[0,1]$. The test set is obtained as a perturbed version of the training set by adding independent Gaussian noise $\varepsilon \sim \mathcal{N}(0,\sigma^2)$ with $\sigma=10^{-2}$ to a fraction of the training data ($30\%$), while leaving the remaining samples unchanged. This construction is closely related to the uniform stability condition \cite{Bousquet2002}.

The same procedure is repeated after projecting both training and test data into wavelet approximation spaces for each level $J \in \{1,2,3,4\}$.
\\

Figure \ref{fig:gengap_bound_vs_n_A} shows the behavior of the empirical generalization gap (\emph{GenGap}) and the theoretical bound $D/\sqrt{n}$ as functions of the sample size $n$, for some values of the parameter $A$. The behavior of the two quantities confirms the theoretical findings.

\begin{figure}[H]
    \centering
    
    \begin{subfigure}{0.45\textwidth}
        \centering
        \includegraphics[width=\linewidth]{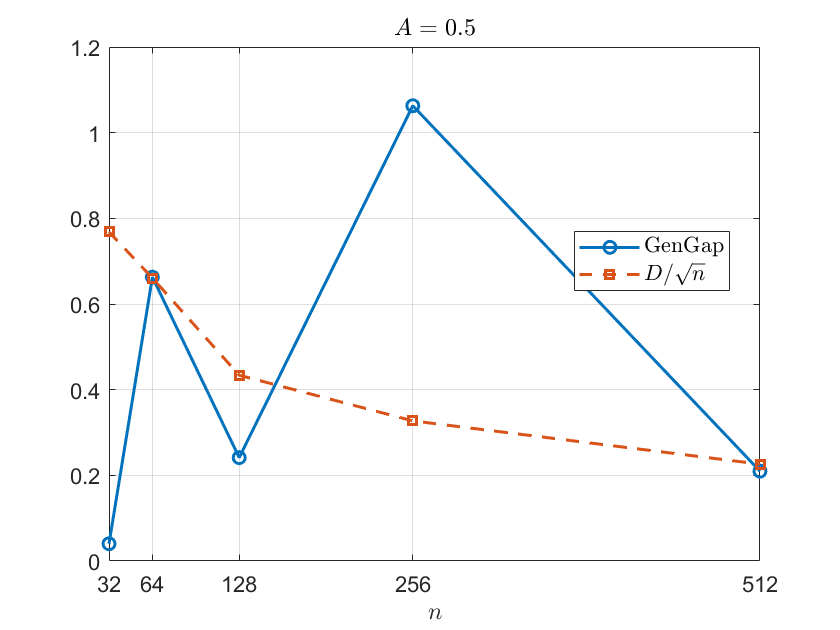}
        \caption{$A = 0.5$}
    \end{subfigure}
    \hfill
    \begin{subfigure}{0.45\textwidth}
        \centering
        \includegraphics[width=\linewidth]{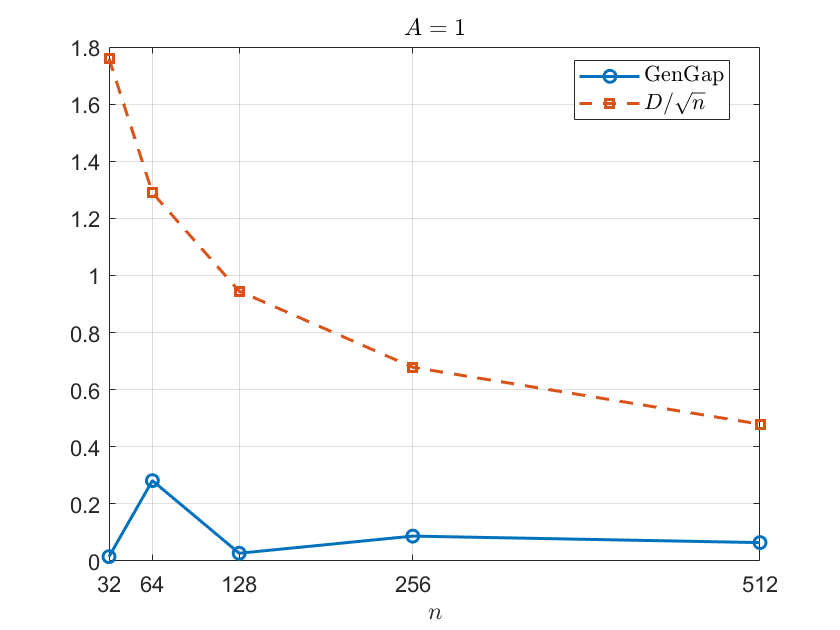}
        \caption{$A = 1$}
    \end{subfigure}
    
    \vspace{0.5cm}
    
    \begin{subfigure}{0.45\textwidth}
        \centering
        \includegraphics[width=\linewidth]{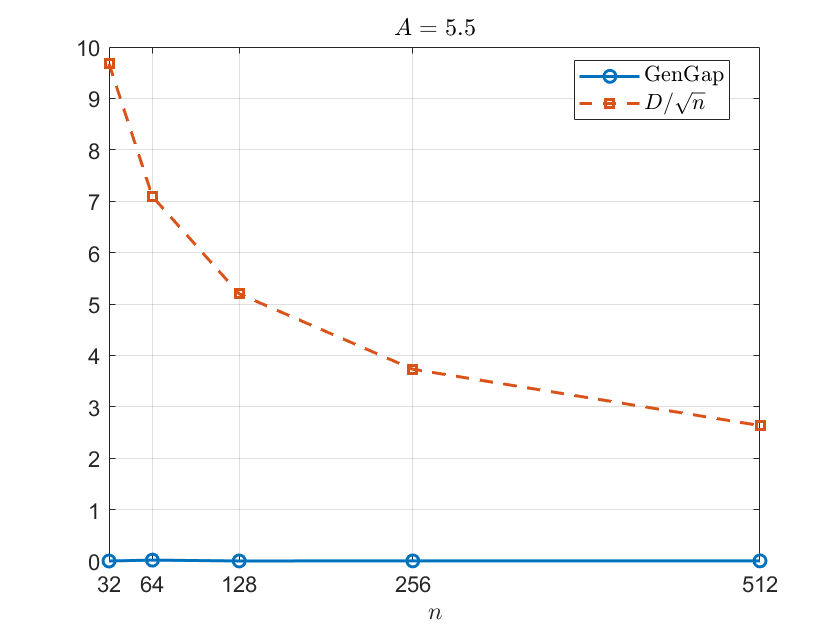}
        \caption{$A = 5.5$}
    \end{subfigure}
    \hfill
    \begin{subfigure}{0.45\textwidth}
        \centering
        \includegraphics[width=\linewidth]{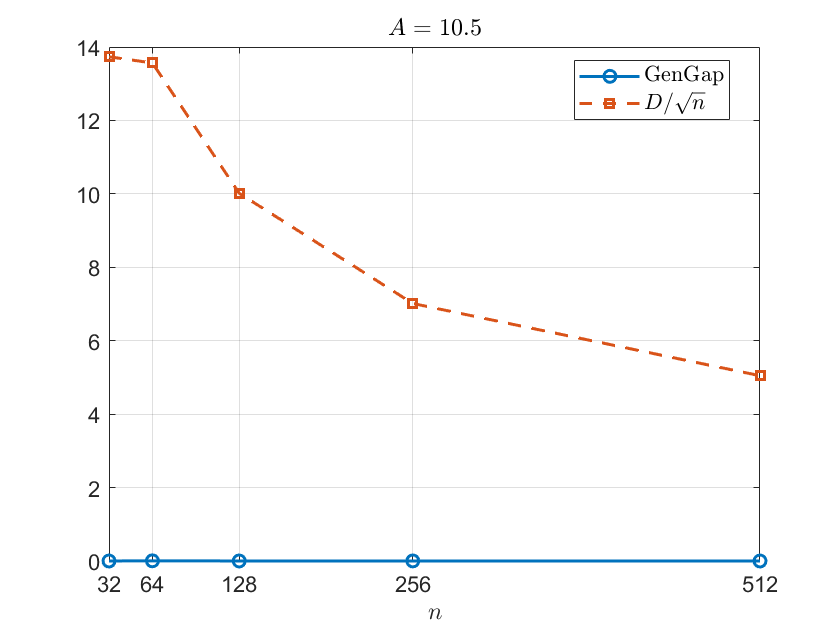}
        \caption{$A = 10.5$}
    \end{subfigure}
    
    \caption{Comparison between the empirical generalization gap (\emph{GenGap}) and the theoretical bound $D/\sqrt{n}$ as a function of the sample size $n$, for different values of the parameter $A \in \{0.5, 1, 5.5, 10\}$. Each panel shows how the empirical gap relates to its corresponding bound as $n$ increases.}
    
    \label{fig:gengap_bound_vs_n_A}
\end{figure}

We extend the previous analysis to the case where a non-expansive wavelet projection $V_J$ is applied to the input data. Figure \ref{fig:GenGapj} illustrates the comparison between empirical and theoretical quantities for each $J$ and $A=1$.

\begin{figure}[H]
    \centering
    
    \begin{subfigure}{0.45\textwidth}
        \centering
        \includegraphics[width=\linewidth]{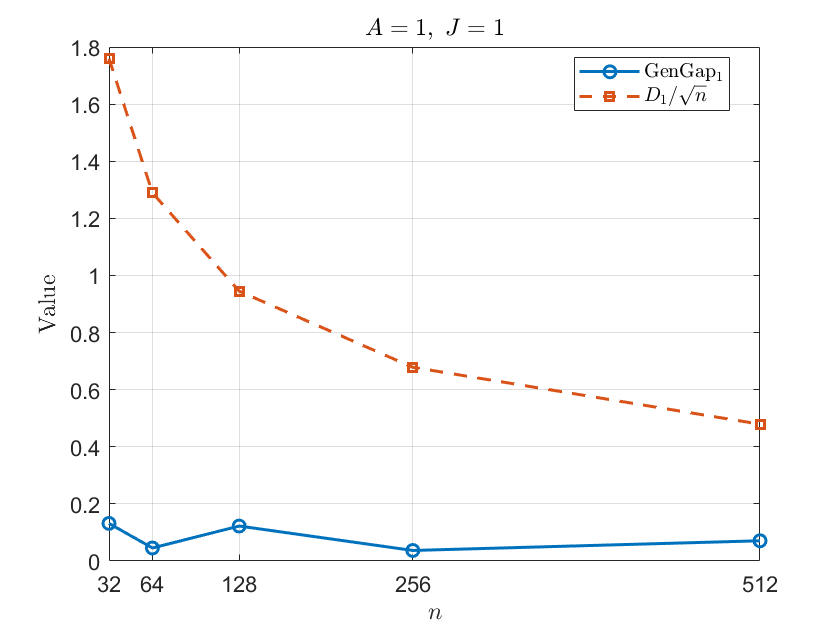}
        \caption{$j = 1$}
    \end{subfigure}
    \hfill
    \begin{subfigure}{0.45\textwidth}
        \centering
        \includegraphics[width=\linewidth]{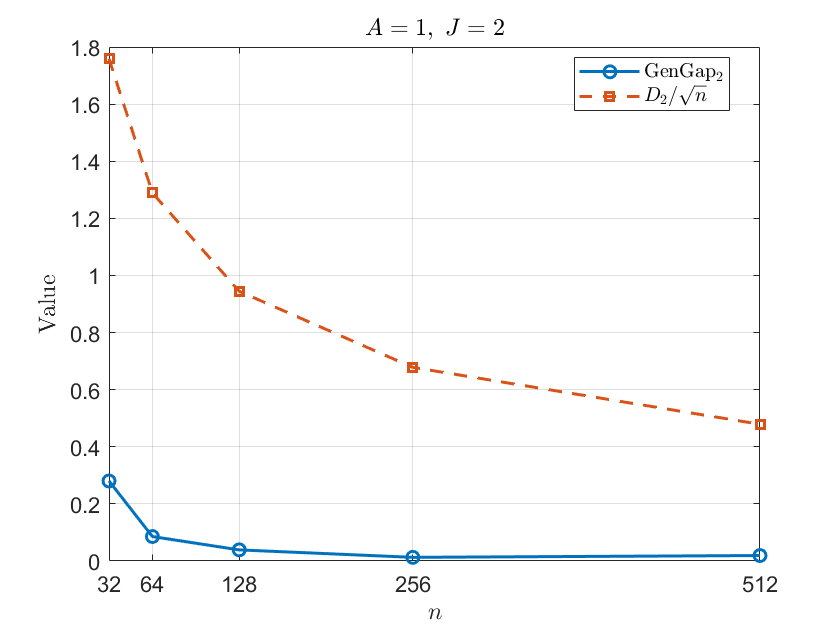}
        \caption{$j = 2$}
    \end{subfigure}
    
    \vspace{0.5cm}
    
    \begin{subfigure}{0.45\textwidth}
        \centering
        \includegraphics[width=\linewidth]{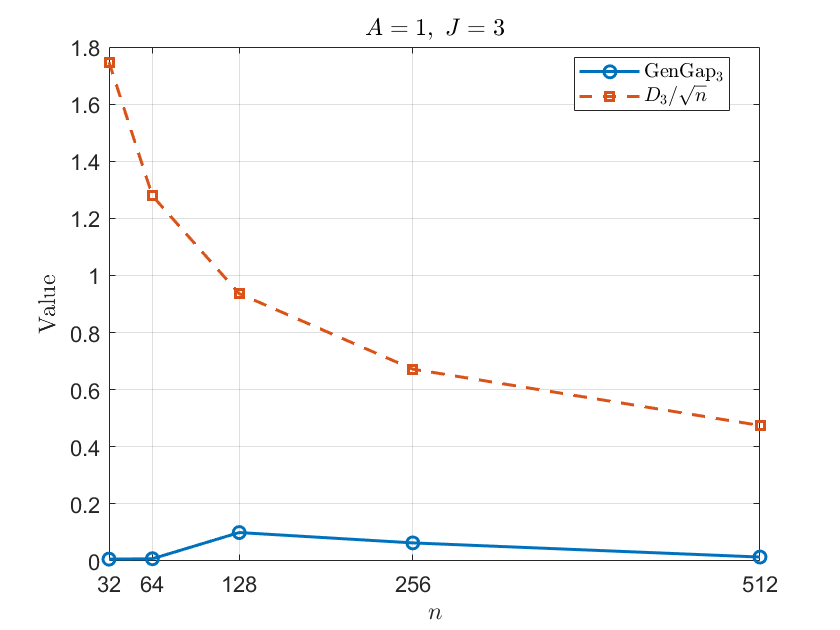}
        \caption{$j = 3$}
    \end{subfigure}
    \hfill
    \begin{subfigure}{0.45\textwidth}
        \centering
        \includegraphics[width=\linewidth]{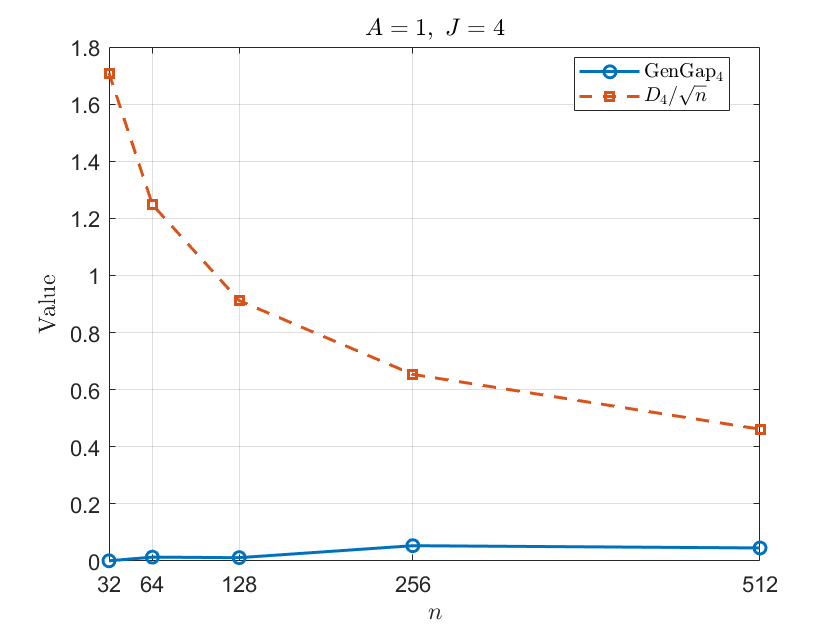}
        \caption{$J = 4$}
    \end{subfigure}
    
    \caption{Comparison between the empirical generalization gap $\mathrm{GenGap}_J$ and corresponding theoretical bound $D_j/\sqrt{n}$ as functions of the sample size $n$, for $J \in \{1,2,3,4\}$ and $A = 1$.}
    
    \label{fig:GenGapj}
\end{figure}

The empirical results indicate that the generalization gap behaves as expected, decreasing with the sample size $n$ and remaining well aligned with the theoretical bounds. The wavelet projections do not compromise stability; in fact, the projected gaps closely follow their respective bounds, which are often smaller than those of the original signals, effectively enhancing stability while preserving the essential statistical structure and filtering out noise. For small values of $A$, where the signal amplitude is low relative to the added noise, the signal-to-noise ratio is reduced, which can slightly affect the alignment with the bound; however, the overall stability and expected behavior remain evident.

\subsection{Comparative assessment of the frequency parameter estimation} 
In this section, we introduce a new set of numerical experiments that aims to assess the effectiveness of the proposed neural network approach to estimate the frequency parameter $\alpha$. The performance of the network is evaluated compared to the algorithm available in the literature \cite{campagnaconticuomo}, which provides an optimal estimate of $\alpha$ through a dedicated optimization procedure.

To thoroughly investigate the generalization capability of the proposed approach, three different scenarios are considered. Each scenario is characterized by a specific choice of training and test datasets, allowing us to progressively increase the level of complexity and mismatch between training and testing conditions. In all scenarios, once the network is trained, the estimation of $\alpha$ is performed on the test data and compared with the corresponding optimal values computed via the reference algorithm. The resulting $\alpha$ values are then used within the HP spline framework to reconstruct the target functions, enabling a direct comparison in terms of both parameter accuracy and reconstruction performance.

In the first scenario, the training data consist of simple exponential functions for which the nominal values of $\alpha$ are known. In \eqref{eq:exp} we report just an example of simple exponential data:
\begin{equation}\label{eq:exp}
  s_1(t) = Ae^{\alpha t}, \qquad s_2(t)=Ate^{-\alpha t}.  
\end{equation}

The test data are constructed by perturbing the same functions with additive Gaussian noise $\varepsilon \sim \mathcal{N}(0,\sigma^2)$ and $\sigma=10^{-2}$. This setting allows us to evaluate the robustness of the neural network with respect to noise while maintaining consistency between the training and testing distributions.

In the second scenario, the training data remains unchanged, still consisting of simple exponential functions with known $\alpha$. However, test data are generated using more complex multi-exponential functions as defined in \eqref{eq:multiexp}, introducing a structural mismatch between training and testing data
\begin{equation}\label{eq:multiexp}
    s_3(t) = \frac{t}{2} e^{-2\alpha} + e^{-\frac{\alpha}{2}t}.
\end{equation}

This scenario is designed to assess the extrapolation capability of the network when applied to functions of increased complexity.

In the third scenario, the training dataset is enriched by including both simple exponential functions (with known $\alpha$) and multi-exponential functions. For the latter, the associated values $\alpha$ are not known a priori but are estimated using the reference optimization algorithm. The test data coincide with the multi-exponential functions considered in the second scenario. This setting allows us to evaluate whether increasing the training set with more representative and heterogeneous data improves the predictive performance of the network.

Figure~\ref{fig:HPcomparison} reports the comparison of predicted vs optimal $\alpha$ for some test data in the three scenarios. Performance is measured in terms of \textit{MSE} and \textit{RE} for the HP-spline reconstruction.

\begin{figure}[H]
    \centering
    
    \begin{subfigure}{0.6\linewidth}
        \centering
        \includegraphics[width=\linewidth]{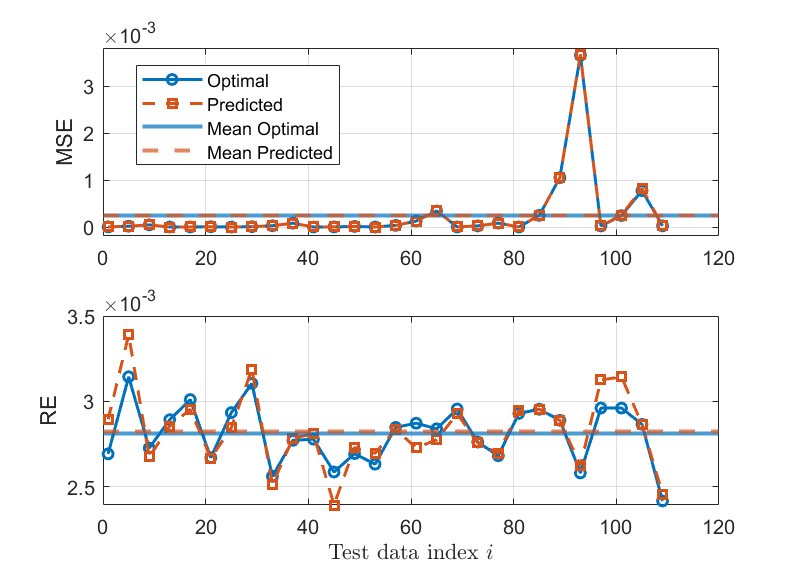}
        \caption{Scenario 1}
    \end{subfigure}
    
    \vspace{0.5cm}
    
    \begin{subfigure}{0.45\linewidth}
        \centering
        \includegraphics[width=\linewidth]{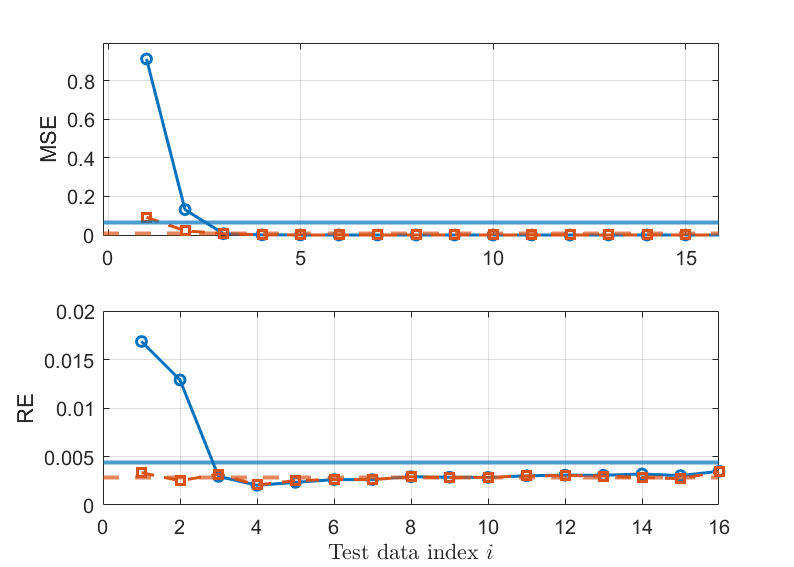}
        \caption{Scenario 2}
    \end{subfigure}
    \hfill
    \begin{subfigure}{0.45\linewidth}
        \centering
        \includegraphics[width=\linewidth]{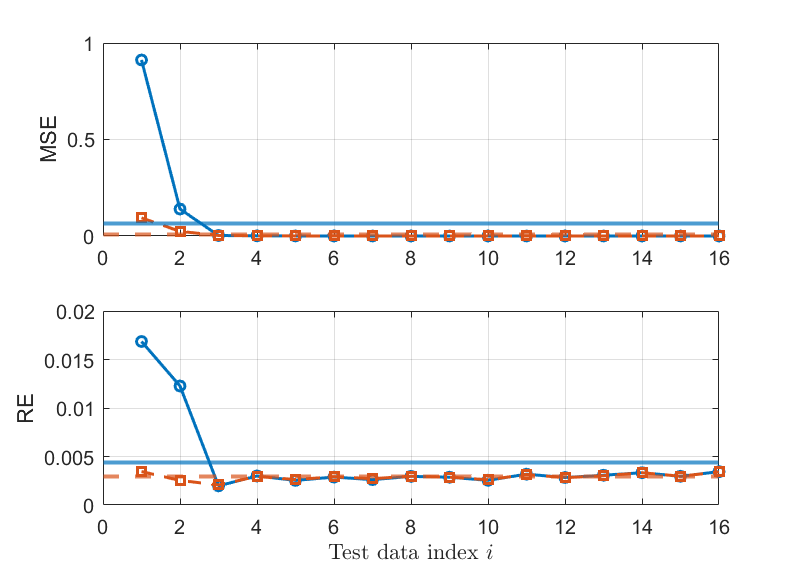}
        \caption{Scenario 3}
    \end{subfigure}
    
    \caption{Comparison of HP-spline regression performance using predicted and optimal $\alpha$ across three scenarios. Solid and dashed lines represent the two methods, while horizontal dotted lines indicate their corresponding mean values. Results are reported in terms of \textit{MSE} and \textit{RE} on a logarithmic scale, computed for each individual test data instance using neural network-based predictions of $\alpha$. HP-spline reconstruction errors are evaluated on the corresponding test signals using uniformly spaced knots with step size \(0.1\).}
    
    \label{fig:HPcomparison}
\end{figure}

Overall, the neural network provides a reliable and computationally efficient estimate of the HP-spline frequency parameter. The predicted $\alpha$ consistently reproduces the regression accuracy of the optimal $\alpha$, even in scenarios with noise or function extrapolation. This validates the approach as a practical alternative to exhaustive search $\alpha$, allowing fast and accurate HP-spline reconstruction in various contexts.

\section*{Conclusions}
In this paper a neural network-based strategy for the estimation of the frequency parameter in HP-splines has been developed and analyzed. The parameter selection problem has been reformulated in a functional setting, enabling its approximation through deep architectures with controlled complexity. Under suitable regularity assumptions, bounds for the induced HP-spline approximation error have been established. In addition, the stability analysis of the whole learning procedure provided further insight into the robustness of the proposed approach, allowing control of the associated generalization gap via uniform stability arguments for the composite model.
The numerical results support the theoretical analysis, demonstrating that neural networks can reliably estimate the frequency parameter in HP-splines while maintaining high reconstruction accuracy and stability even in  scenarios with reduced smoothness. The results also indicate that moderately sized networks are sufficient for practical applications, and network design can remain flexible without compromising accuracy.
These findings suggest that learning-based parameter selection constitutes a viable alternative to classical approaches. Further investigation is required in higher-dimensional settings, where tensor-product constructions and directional parameter estimation significantly increase computational complexity. Refinements of approximation error bounds in more complex scenarios also remain an open direction for future research.

  \section*{Acknowledgements}
The authors are members of the INdAM research group GNCS, which has
partially supported this work. This research has been accomplished within the UMI-TAA group, of which      the   authors are members, and UMI AI\&ML\&MAT group. RC was partially supported by the Italian MUR through the PRIN2022 Project 'Numerical Optimization with Adaptive Accuracy and Applications to Machine Learning',   code: 2022N3ZNAX,  and 
the PRIN2022-PNRR Project 'A multidisciplinary approach to evaluate ecosystems resilience under climate change',   code: P2022WC2ZZ. VB and DV have been partially supported by PNRR-CN1 SPOKE 6 - Multiscale modeling engineering applications B83C22002940006.

 \section*{Declarations}
The MATLAB code used in this study is available from the authors upon request.

\end{document}